\pdfoutput=1
\RequirePackage{ifpdf}
\ifpdf % We~are running pdfTeX in pdf mode
\documentclass[pdftex]{sigma}
\else
\documentclass{sigma}
\fi

\usepackage{tikz}

\numberwithin{equation}{section}

\newtheorem{Theorem}{Theorem}[section]
\newtheorem{Corollary}[Theorem]{Corollary}
\newtheorem{Proposition}[Theorem]{Proposition}
{ \theoremstyle{definition}

\newtheorem{Example}[Theorem]{Example}
\newtheorem{Remark}[Theorem]{Remark} }

\newcommand{\A}{\mathcal{A}}
\newcommand{\B}{\mathcal{B}}
\newcommand{\Z}{\mathbb{Z}}

\newcommand{\C}{\mathbb{C}}

\newcommand{\Cat}{\operatorname{Cat}}
\newcommand{\Shi}{\operatorname{Shi}}
\newcommand{\Der}{\operatorname{Der}}

\newcommand{\Ker}{\operatorname{Ker}}

\newcommand{\sint}{\sum\nolimits}

\begin{document}
%\allowdisplaybreaks

\newcommand{\arXivNumber}{2009.13710}

\renewcommand{\thefootnote}{}

\renewcommand{\PaperNumber}{038}

\FirstPageHeading

\ShortArticleName{The Primitive Derivation and Discrete Integrals}

\ArticleName{The Primitive Derivation and Discrete Integrals\footnote{This paper is a~contribution to the Special Issue on Primitive Forms and Related Topics in honor of~Kyoji Saito for his 77th birthday. The full collection is available at \href{https://www.emis.de/journals/SIGMA/Saito.html}{https://www.emis.de/journals/SIGMA/Saito.html}}}

\Author{Daisuke SUYAMA~$^{\rm a}$ and Masahiko YOSHINAGA~$^{\rm b}$}

\AuthorNameForHeading{D.~Suyama and M.~Yoshinaga}

\Address{$^{\rm a)}$~Faculty of Integrated Media, Wakkanai Hokusei Gakuen University,\\
\hphantom{$^{\rm a)}$}~1-2290-28 Wakabadai, Wakkanai, Hokkaido 097-0013, Japan}
\EmailD{\href{mailto:suyama@wakhok.ac.jp}{suyama@wakhok.ac.jp}}

\Address{$^{\rm b)}$~Department of Mathematics, Faculty of Science, Hokkaido University,\\
\hphantom{$^{\rm b)}$}~Kita 10, Nishi 8, Kita-Ku, Sapporo 060-0810, Japan}
\EmailD{\href{mailto:yoshinaga@math.sci.hokudai.ac.jp}{yoshinaga@math.sci.hokudai.ac.jp}}

\ArticleDates{Received September 30, 2020, in final form April 09, 2021; Published online April 13, 2021}

\Abstract{The modules of logarithmic derivations for the (extended) Catalan and Shi arrangements associated with root systems are known to be free. However, except for a~few cases, explicit bases for such modules are not known. In this paper, we construct explicit bases for type $A$ root systems. Our construction is based on Bandlow--Musiker's integral formula for a basis of the space of quasiinvariants. The integral formula can be considered as an expression for the inverse of the primitive derivation introduced by K.~Saito. We~prove that the discrete analogues of the integral formulas provide bases for Catalan and Shi arrangements.}


\Classification{52C35; 20F55}

\begin{flushright}
\begin{minipage}{60mm}
\it
Dedicated to Professor Kyoji Saito\\ on the occasion of his 77th birthday
\end{minipage}
\end{flushright}

\renewcommand{\thefootnote}{\arabic{footnote}}
\setcounter{footnote}{0}

\section{Introduction}

\subsection{Background}%\label{subsec:backgr}

Let $V$ be an $\ell$-dimensional linear space over $\C$.
Let $S=S(V^*)$ be the set of polynomial functions on $V$.
Choose $x_1, \dots, x_\ell$ as a basis of $V^*$ and identify $S$ with the
polynomial ring $\C[x_1, \dots, x_\ell]$.
Let $\Der_S=\bigoplus_{i=1}^\ell S\partial_i$ be the module of polynomial
vector fields ($\C$-linear derivations of $S$).

Let $\A=\{H_1, \dots, H_n\}$ be a central arrangement of hyperplanes.
For each $H\in\A$, choose a~linear form $\alpha_H\in V^*$ such that
$H=\Ker(\alpha_H)$.
For each nonnegative integer $m$, let
\begin{gather*}
D(\A, m):=\big\{\delta\in\Der_S\mid \delta\alpha_H\in(\alpha_H^m),
\mbox{ for any }H\in\A\big\}.
\end{gather*}
If $m=1$, $D(\A, 1)$ is simply denoted by $D(\A)$, whose elements are
called logarithmic derivations. The module $D(\A)$ was
introduced in~\cite{sai-log} for the purpose of computing Gauss--Manin
connections, and $D(\A, m)$ was introduced by Ziegler~\cite{zie} for studying
restrictions of free arrangements.
The~alge\-b\-raic structures of these modules are thought to reflect
the combinatorial nature of~$\A$ (see~\cite{ot, yos-survey}).

Now we assume that $\A$ is a Coxeter arrangement, that is,
the set of reflecting hyperplanes of a finite irreducible real reflection group
$W\subset {\rm GL}(V)$.
The definition of $D(\A, m)$ naturally gives rise to a filtration:
\begin{gather}
\label{eq:filt}
\Der_S=D(\A, 0)\supset D(\A, 1)\supset D(\A, 2)\supset\cdots.
\end{gather}
The filtration~(\ref{eq:filt}) is closely related to several important structures.
First, taking $W$-invariant parts, we have
\begin{gather*}
%\label{eq:Hodgefilt}
D(\A, 0)^W= D(\A, 1)^W\supset D(\A, 2)^W= D(\A, 3)^W\supset\cdots.
\end{gather*}
This filtration is known to be equal to the semi-infinite Hodge filtration studied by
K.~Saito~\cite{sai-lin, sai-unif, ter-hodge, yos-prim},
which is a crucial structure in his theory
of primitive forms~\cite{sai-prim}.
In particular, the inverse operator of the
so-called \emph{primitive derivation} $\nabla_D$ describes the filtration as
\begin{gather*}
%\label{eq:primder}
D(\A, 2m+1)^W=\big(\nabla_D^{-1}\big)^mD(\A)^W.
\end{gather*}
As indicated by Misha Feigin (see forthcoming paper~\cite{aefy} for details),
these spaces are also isomorphic
to the isotypic component of the spaces of $m$-quasiinvariants, which
were introduced in the study of the Calogero--Moser system
\cite{bm, cha-ves, fei-ves}.

Around 2000, Terao proved that the module $D(\A, m)$ is an
$S$-free module using these structures~\cite{ter-mult}.
Terao's results on the freeness of $D(\A, m)$ opened new perspectives
between the primitive derivation and enumerative combinatorics of
Catalan/Shi arrangements.

Catalan and Shi arrangements are classes of finite truncations of
affine Weyl arrangements for root systems.
The terminology ``Catalan arrangement''
is explained by the fact that the number of chambers in the fundamental
region of a type $A$ root system is equal to the Catalan number~\cite{ps-def}.
The Shi arrangement was introduced by J.-Y.~Shi in~\cite{shi-kl}
in the study of affine Weyl groups.
In 1996, Edelman--Reiner~\cite{ede-rei} posed a conjecture concerning the
freeness of cones of Catalan and Shi arrangements for root systems.
This conjecture was proved for type $A$ root system
by Edelman--Reiner~\cite{ede-rei} and
Athanasiadis~\cite{ath-free}, and later for all root systems by~\cite{yos-char}.
The~freeness of $D(\A, m)$ played crucial role in the proof of
\cite{yos-char} because $D(\A, m)$ can be regarded as the ``leading terms''
of the logarithmic vector fields for Catalan and Shi arrangements.
In the present paper we focus on the construction of explicit bases for
these modules.

\subsection{Constructions of explicit bases}%\label{subsec:const}

In this section, we introduce Catalan and Shi arrangements (of type $A_{\ell-1}$). We~define $m$-Catalan arrangement $\Cat_\ell(m)$ as
\begin{gather*}
\prod_{\substack{1\leq i<j\leq \ell\\ -m\leq k\leq m}}(x_i-x_j-k)=0,
\end{gather*}
and $m$-Shi arrangement $\Shi_\ell(m)$ by
\begin{gather*}
\prod_{1-m\leq k\leq m}\prod_{1\leq i<j\leq \ell}(x_i-x_j-k)=0.
\end{gather*}
We also denote $\Cat_\ell(0)$ by $\B_\ell$. The arrangement $\B_\ell$ is
defined as the polynomial $\prod_{1\leq i<j\leq \ell}(x_i-x_j)$ and
called the braid arrangement.

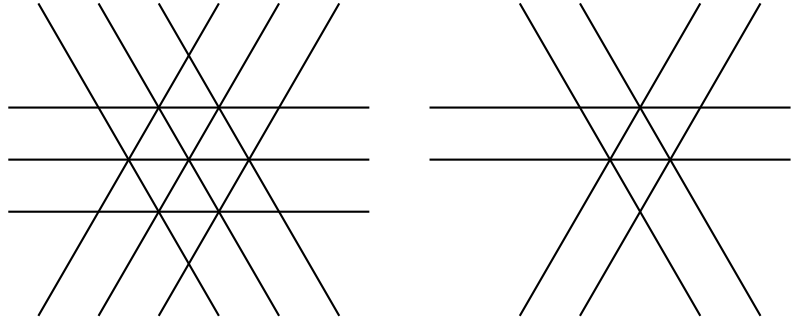
\begin{figure}[htbp]
\centering
\begin{tikzpicture}[scale=0.8]

\draw[thick, black] (0, 2.134) -- +(6,0);
\draw[thick, black] (0, 3) -- +(6,0);
\draw[thick, black] (0, 3.866) -- +(6,0);

\draw[thick, black] (3, 3) -- +(60:3);
\draw[thick, black] (3, 3) -- +(240:3);

\draw[thick, black] (2, 3) -- +(60:3);
\draw[thick, black] (2, 3) -- +(240:3);

\draw[thick, black] (4, 3) -- +(60:3);
\draw[thick, black] (4, 3) -- +(240:3);

\draw[thick, black] (3, 3) -- +(300:3);
\draw[thick, black] (3, 3) -- +(120:3);

\draw[thick, black] (2, 3) -- +(300:3);
\draw[thick, black] (2, 3) -- +(120:3);

\draw[thick, black] (4, 3) -- +(300:3);
\draw[thick, black] (4, 3) -- +(120:3);

\draw[thick, black] (7, 3) -- +(6,0);
\draw[thick, black] (7, 3.866) -- +(6,0);

\draw[thick, black] (10, 3) -- +(60:3);
\draw[thick, black] (10, 3) -- +(240:3);

\draw[thick, black] (11, 3) -- +(60:3);
\draw[thick, black] (11, 3) -- +(240:3);

\draw[thick, black] (10, 3) -- +(300:3);
\draw[thick, black] (10, 3) -- +(120:3);

\draw[thick, black] (11, 3) -- +(300:3);
\draw[thick, black] (11, 3) -- +(120:3);

\end{tikzpicture}
\caption{$\Cat_3(1)$ and $\Shi_3(1)$ (the intersections with the plane
$x_1+x_2+x_3=0$ are drawn).}
%\label{fig:catshi}
\end{figure}

The cones $c\Cat_\ell(m)$ and $c\Shi_\ell(m)$ are defined by the homogeneous
polynomials
\begin{gather*}
z\!\!\! \prod_{\substack{1\leq i<j\leq \ell\\ -m\leq k\leq m}}(x_i-x_j-kz)=0,
\end{gather*}
and %$\Shi_\ell(m)$ by
\begin{gather*}
z\!\!\! \prod_{1-m\leq k\leq m}\prod_{1\leq i<j\leq \ell}(x_i-x_j-kz)=0,
\end{gather*}
respectively.
As we have already noted, the modules
$D(c\Cat_\ell(m))$ and $D(c\Shi_\ell(m))$ are free. The exponents, that
is, the degrees of the homogeneous bases of the modules, are as follows
\begin{gather*}
\exp(c\Cat_\ell(m))=\{0, 1, m\ell+1, m\ell+2, \dots, m\ell+\ell-1\},
\\
\exp(c\Shi_\ell(m))=\{0, 1, \underbrace{m\ell, m\ell, \dots, m\ell}_{\ell-1}\}.
\end{gather*}
We note that explicit bases were not constructed in the known proofs.
Indeed, the proofs by Edelman--Reiner~\cite{ede-rei} and
Athanasiadis~\cite{ath-free} used Terao's addition--deletion theorem
of freeness \cite[Theorem 4.51]{ot}, and that in~\cite{yos-char}
used cohomological arguments to guarantee the existence
of~glo\-bal sections of certain coherent sheaves associated with
the graded module $D(\A)$.
Since then, a~number of efforts have been made to construct explicit bases for
$D(c\Cat_\ell(m))$ and~$D(c\Shi_\ell(m))$.
First, in~\cite{su-te}, a basis for $D(c\Shi_\ell(1))$ was constructed
using the Bernoulli polynomial. Subsequently, in~\cite{gpt} and~\cite{su-bc},
similar bases were constructed for root systems of~type $B$, $C$, and $D$.
Note that these works are for Shi arrangements with $m=1$.
Catalan arran\-gements and Shi arrangements with $m>1$ have not been covered.
For larger $m$, the type
$A_2$ was the only known case. Namely, explicit bases were constructed
for $c\Cat_3(m)$ and~$c\Shi_3(m)$, $m\geq 1$, in~\cite{abe-suy}.

Our purpose in the present paper is to construct an explicit basis
for $D(c\Cat_\ell(m))$ and $D(c\Shi_\ell(m))$, for all $\ell\geq 2$ and $m\geq 1$.
The paper is organized as follows.
The starting point of our study is Bandlow--Musiker's integral expression
\cite{bm} (which goes back to Felder--Veselov's integral expression~\cite{fel-ves})
for a basis of the space of quasiinvariants introduced in
\cite{cha-ves, fei-ves}. Misha Feigin~\cite{feigin} communicated to us
that Bandlow--Musiker's formula provides a basis for the multiarrangement
$D(\B_\ell, 2m+1)$.
More precisely, for appropriate choices
of $k$, the integral expressions
\begin{gather}\label{eq:introbm}
\sum_{i, j=1}^\ell
\bigg(\int_{x_i}^{x_j}t^k\bigg(\prod_{p=1}^\ell(t-x_p)\bigg)^m {\rm d}t\bigg)\partial_i
\end{gather}
provide a basis for $D(\B_\ell, 2m+1)$. We~discuss these facts
in addition to a basis for $D(\B_\ell, 2m)$ in~Section~\ref{sec:bm}.

After introducing the notion of ``discrete integrals'' in Section~\ref{sec:discreteint},
we present the main results in Section~\ref{sec:main}, that is, we prove the
following: a basis for $D(c\Cat_\ell(m))$ is obtained from~(\ref{eq:introbm}) by~simply replacing the integration ``$\int_a^b {\rm d}t$''
with the discrete integration ``$\sint_a^b\Delta t$'' as follows
\begin{gather*}
%\label{eq:introdiscrete}
\sum_{i, j=1}^\ell
\bigg(\sint_{x_i}^{x_j}t^k\bigg(\prod_{p=1}^\ell(t-x_p)\bigg)^{\underline{m}} \Delta t\bigg)\partial_i.
\end{gather*}
(To be precise, we need to homogenize the above polynomial vector field,
see Sections~\ref{sec:discreteint} and~\ref{sec:main} for notations and details.)
We also provide a basis for $D(c\Shi_\ell(m))$.

\section{Bandlow--Musiker's expression}\label{sec:bm}

Recall that $\B_\ell$ denotes the braid arrangement defined by
$\prod_{1\leq i<j\leq\ell}(x_i-x_j)$. The symmetric group
$W=\mathfrak{S}_\ell$ naturally acts on $\B_\ell$ by the permutation of
coordinates.
In this section, we~construct a basis for $D(\B_\ell, m)$, $m\geq 1$.
Note that because the vector field
\begin{gather*}
\theta_0:=\partial_1+\partial_2+\dots+\partial_\ell
\end{gather*}
annihilates the linear form $x_i-x_j$, we have $\theta_0\in D(\B_\ell, m)$
for any $m$.

We also set
$g(t):=(t-x_1)(t-x_2)\cdots(t-x_\ell)\in \C[t, x_1, \dots, x_\ell]$.
Following Bandlow--Musiker~\cite{bm} and Feigin~\cite{feigin},
for $m, k\geq 0$, we introduce the following vector field
\begin{gather}
\label{eq:eta}
\eta_k^m=\sum_{i=1}^\ell\bigg(\sum_{j=1}^\ell\int_{x_i}^{x_j}t^kg(t)^m {\rm d}t\bigg)\partial_i.
\end{gather}
\begin{Proposition}[\cite{bm, feigin}]\label{prop:bm}
The vector fields
$\eta_0^m, \eta_1^m, \dots, \eta_{\ell-2}^m, \theta_0$ form a
basis for $D(\B_\ell, 2m+1)^W$ as an $S^W$-module.
\end{Proposition}

Recall that Terao~\cite{ter-mult} proved that there exists a $W$-invariant basis
for $D(\B_\ell, 2m+1)$. Terao's invariant basis generates a submodule
of $D(\B_\ell, 2m+1)^W$ over $S^W$. However, these must coincide since
the degrees of Terao's basis are equal to those of the above basis.
Thus we have the following.
\begin{Corollary}
\label{cor:multibasis}
The vector fields
$\eta_0^m, \eta_1^m, \dots, \eta_{\ell-2}^m, \theta_0$ form a
basis of $D(\B_\ell, 2m+1)$ as an~$S$-module.
\end{Corollary}
Here we give a direct proof of Corollary~\ref{cor:multibasis} in order to see
the relationship between integral expression~(\ref{eq:eta}) and
the primitive derivation.
First, we prove $\eta_k^m\in D(\B_\ell, 2m+1)$. Since $\eta_k^m$ is
$W$-symmetric, it is sufficient to show that
\begin{gather*}
\eta_k^m(x_1-x_2)=\ell\int_{x_1}^{x_2}t^kg(t)^m{\rm d}t
\end{gather*}
is divisible by $(x_1-x_2)^{2m+1}$. This can be checked by
%a straightforward calculation using
the change of variables $t'=t-x_1$.
Indeed, since the integrand is divisible by $(t-x_1)^m(t-x_2)^m$, after the
change of variable, it is divisible by $(t')^m(t'-(x_2-x_1))^m$. Therefore,
the definite integral $\int_0^{x_2-x_1}{\rm d}t'$ is divisible by $(x_2-x_1)^{2m+1}$.

\begin{Example}
When $\ell=2$, $m=1$, we obtain the following famous formula
\begin{gather*}
\int_{x_1}^{x_2}(t-x_1)(t-x_2)\,{\rm d}t=\frac{(x_1-x_2)^3}{6}.
\end{gather*}
\end{Example}

Next we describe the action of the primitive derivation on the vector field $\eta_k^m$.
Let $\nabla$ denote the integrable connection with flat sections
$\partial_1, \dots, \partial_\ell$. More explicitly, for polynomial
vector fields $\delta$ and $\eta=\sum_{i=1}^\ell f_i\partial_i$, we define
\begin{gather*}
\nabla_\delta\eta=\sum_{i=1}^\ell (\delta f_i)\,\partial_i.
\end{gather*}
Let $P_i$ denote the coefficient of $t^{\ell-i}$ in $g(t)$, that is,
$g(t)=t^\ell+P_1t^{\ell-1}+\dots+P_\ell$. Then $P_1, \dots, P_\ell
\in\Z[x_1, \dots, x_\ell]$ are elementary symmetric functions (up to the sign),
and satisfy $\C[x_1, \dots, x_\ell]^W=\C[P_1, \dots, P_\ell]$. Therefore,
we can regard $P_1, \dots, P_\ell$ as a system of coordinates on the
quotient space $\C^\ell/W$. The vector field $D=\frac{\partial}{\partial P_\ell}$
is called the primitive derivation~\cite{sai-lin, sai-unif}. The lift of $D$ by
the natural projection $\pi\colon \C^\ell\longrightarrow\C^\ell/W$ is
expressed as
\begin{gather*}
%\label{eq:prim}
D=
\frac{1}{Q}\det
\begin{pmatrix}
\frac{\partial P_1}{\partial x_1}&
\frac{\partial P_2}{\partial x_1}&
\dots &
\frac{\partial P_{\ell-1}}{\partial x_1}&
\frac{\partial}{\partial x_1}
\\[1ex]
\frac{\partial P_1}{\partial x_2}&
\frac{\partial P_2}{\partial x_2}&
\dots &
\frac{\partial P_{\ell-1}}{\partial x_2}&
\frac{\partial}{\partial x_2}
\\
\vdots&\vdots&\ddots&\vdots&\vdots
\\
\frac{\partial P_1}{\partial x_\ell}&
\frac{\partial P_2}{\partial x_\ell}&
\dots &
\frac{\partial P_{\ell-1}}{\partial x_\ell}&
\frac{\partial}{\partial x_\ell}
\end{pmatrix}\!,
\end{gather*}
where $Q=\det\big(\frac{\partial P_i}{\partial x_j}\big)$ is the Jacobian of the projection
$\pi$. For simplicity, we also denote by $D$ the lift of the primitive derivation
$\frac{\partial}{\partial P_\ell}$ . Note that other vector fields
$\frac{\partial}{\partial P_1}, \frac{\partial}{\partial P_2}, \dots,
\frac{\partial}{\partial P_{\ell-1}}$ also have similar expressions.
The definition~(\ref{eq:eta}) can be written as
\begin{gather*}
%\label{eq:eta02}
\eta_k^m=\sum_{i=1}^\ell\bigg(\sum_{j=1}^\ell
\int_{x_i}^{x_j}t^k\big(t^\ell+P_1t^{\ell-1}+\dots +P_\ell\big)^m {\rm d}t\bigg)\partial_i.
\end{gather*}
Then, if $m>0$, differentiation by $\frac{\partial}{\partial P_j}$ yields
(see Remark~\ref{rem:details} for details)
\begin{gather}
\label{eq:bibun}
\nabla_{\frac{\partial}{\partial P_j}}\eta_k^m=m\,\eta_{k+\ell-j}^{m-1}.
\end{gather}
When $m=0$, it is checked by Saito's criterion~\cite{ot,sai-log} that
the vector fields
$\eta_0^0, \eta_1^0, \dots, \eta_{\ell-2}^0, \theta_0$ form a
basis of $D(\B_\ell, 1)$.
For $m>0$, we have $\nabla_D^m\eta_k^m=m!\cdot\eta_k^0$.
Recall that
$D(\B_\ell\cap H_0$, $2m+1)^W=\nabla_D^{-m}D(\B_\ell\cap H_0, 1)^W$
\cite[Corollary~10]{yos-prim}. Therefore,
$\nabla_D^{-m}\eta_0^0, \dots, \nabla_D^{-m}\eta_{\ell-1}^0$ form a~basis of $D(\B_\ell\cap H_0, 2m+1)$, where $H_0$ is the hyperplane
$x_1+\dots + x_\ell=0$. Adding $\theta_0$, we~obtain a basis for
$D(\B_\ell, 2m+1)$.
This completes the proof of Corollary~\ref{cor:multibasis}.

Now we consider $D(\B_\ell, 2m)$. Let
$g_k(t):=\frac{g(t)}{(t-x_k)}=(t-x_1)\cdots \widehat{(t-x_k)}\cdots (t-x_\ell)$
for $1\leq k\leq \ell$.
For $m>0$ and $k=1, \dots, \ell$, we define the vector field $\sigma_k^m$ as
\begin{gather}
\label{eq:sigma1}
\sigma_k^m:=\sum_{i=1}^\ell\bigg(\sum_{j=1}^\ell\int_{x_i}^{x_j}g(t)^{m-1}g_k(t)\, {\rm d}t \bigg)\partial_i.
\end{gather}

\begin{Proposition}
\label{prop:simpleroots}
Vector fields
$\theta_0, \sigma_1^m-\sigma_2^m, \sigma_2^m-\sigma_3^m, \dots, \sigma_{\ell-1}^m-\sigma_\ell^m$ form a basis of $D(\B_\ell, 2m)$.
\end{Proposition}

\begin{proof}
We first note that $\frac{\partial g(t)}{\partial x_k}=-g_k(t)$. Hence we have
\begin{gather}
\label{eq:differentiate2}
\nabla_{\partial_k}\eta_0^m=-m\sigma_k^m,
\end{gather}
and
\begin{gather*}
\nabla_{(\partial_i-\partial_{i+1})}\eta_0^m=-m\big(\sigma_i^m-\sigma_{i+1}^m\big).
\end{gather*}
Since $\partial_1-\partial_2, \dots, \partial_{\ell-1}-\partial_\ell$
form a basis of $H_0$, by~\cite[Theorem 7]{yos-prim},
$\theta_0,\! \nabla_{\!(\partial_1-\partial_2)}\eta_0^m, \dots,\!
\nabla_{\!(\partial_{\ell-1}-\partial_\ell)}\eta_0^m$
form a basis of $D(\B_\ell, 2m)$.
\end{proof}

\begin{Remark}
\label{rem:details}
Recall that for certain functions $a(x)$, $b(x)$, $f(x, t)$, we have
\begin{gather*}
\frac{{\rm d}}{{\rm d}x}\int_{a(x)}^{b(x)}f(x, t)\,{\rm d}t=
\frac{{\rm d}b(x)}{{\rm d}x}f(x, b(x))-\frac{{\rm d}a(x)}{{\rm d}x}f(x, a(x))+
\int_{a(x)}^{b(x)}\frac{\partial f(x, t)}{\partial x}{\rm d}t.
\end{gather*}
In our setting, since the integrand $t^kg(t)^m$ vanishes at
$x_i$, the first two terms do not contribute. Thus we have
the equations~(\ref{eq:bibun}) and~(\ref{eq:differentiate2}).
\end{Remark}

\section{Discrete integrals}
\label{sec:discreteint}

In this section we only consider polynomial functions.
For a function $f(t)$, we define the difference operator $\Delta$ as
$\Delta f(t)=f(t+1)-f(t)$. When $\Delta F(t)=f(t)$,
$F(t)$ is called an indefinite summation (or antidifference)
of $f(t)$, and denoted by
\begin{gather*}
F(t)=\sint f(t)\Delta t.
\end{gather*}
Let $F(t)$ be an indefinite summation of $f(t)$. Then
we define the definite summation as
\begin{gather*}
\sint_a^b f(t)\Delta t=F(b)-F(a).
\end{gather*}
Obviously we have the following.
\begin{gather*}
\sint_b^af(t)\Delta t=-\sint_a^bf(t)\Delta t,
\\
\sint_a^cf(t)\Delta t=\sint_a^bf(t)\Delta t + \sint_b^cf(t)\Delta t.
\end{gather*}
Note that if $b-a=n$ is a positive integer, the definite summation
is nothing but the finite sum
\begin{gather}
\label{sumformula}
\sint_a^b f(t)\Delta t=f(a)+f(a+1)+\cdots+f(b-1).
\end{gather}

\begin{Example}
Recall that the Bernoulli polynomial $B_n(t)$ is a monic polynomial
with rational coefficients defined by
\begin{gather}
\label{eq:Bpoly}
\sum_{n=0}^\infty
\frac{B_n(x)}{n !}t^n
=\frac{t{\rm e}^{xt}}{{\rm e}^t-1},
\end{gather}
$\big($e.g., $B_0(x)=1$, $B_1(x)=x-\frac{1}{2}$, $B_2(x)=x^2-x+\frac{1}{6}$,
$B_3(x)=x^3-\frac{3}{2}x^2+\frac{1}{2}x$,
$B_4(x)=x^4-2x^3+x^2-\frac{1}{30}$, $\ldots\big)$.
By applying the difference operator with respect to the variable $x$ to
the equation~(\ref{eq:Bpoly}), we have
\begin{gather*}
\sum_{n=0}^\infty
\frac{\Delta B_n(x)}{n !}t^n
=t{\rm e}^{xt}=\sum_{n=1}^\infty\frac{x^{n-1}}{(n-1)!}t^n.
\end{gather*}
Thus we have $\Delta B_n(t)=nt^{n-1}$.
Therefore, the monomial $x^n$ has an indefinite summation $\frac{B_{n+1}(x)}{n+1}$.
Furthermore, an arbitrary polynomial $f(x)=\sum_i a_ix^i$ has an indefinite summation
$\sum_ia_i\frac{B_{i+1}(x)}{i+1}$.
\end{Example}
The leading part of a definite summation is equal to a definite integral.
More precisely, we~have the following.

\begin{Proposition}\label{prop:leading}
Let $f(x_0, x_1, \dots, x_n)\in\C[x_0, \dots, x_n]$ be a homogeneous
polynomial of deg\-ree~$d$. Let
\begin{gather*}
F(y_1, y_2, x_1, \dots, x_n):=
\sint_{y_1}^{y_2}f(x_0, x_1, \dots, x_n)\Delta x_0.
\end{gather*}
Then $F(y_1, y_2, x_1, \dots, x_n)$ is a $($not necessarily homogeneous$)$
polynomial of degree $d+1$ in $y_1, y_2, x_1, \dots, x_n$
whose highest degree part is
\begin{gather*}
\lim_{z\to 0} z^{d+1}F\bigg(\frac{y_1}{z},\frac{y_2}{z},\frac{x_1}{z}, \dots,\frac{x_n}{z}\bigg)
=\int_{y_1}^{y_2}f(x_0, x_1, \dots, x_n)\, {\rm d} x_0.
\end{gather*}
\end{Proposition}

\begin{proof}
This is straightforward from the fact that the leading term of
$\frac{B_{n+1}(t)}{n+1}$ is $\frac{t^{n+1}}{n+1}$,
which is an indefinite integral of $t^n$.
\end{proof}

The following is a discrete analogue of the power.
Let $n>0$ be a positive integer. We~define the falling power
$f(t)^{\underline{n}}$ as
\begin{gather*}
f(t)^{\underline{n}}=f(t)f(t-1)\cdots f(t-n+1).
\end{gather*}

\section{Main results}\label{sec:main}

\subsection{A basis for the Catalan arrangement}

Let $\delta=\sum_{i=1}^\ell f_i(x_1, \dots, x_\ell)\,\partial_i$ be a polynomial
vector field ($f_1, \dots, f_\ell$ are not necessarily homogeneous).
Let $d:=\max\{\deg f_1, \dots, \deg f_\ell\}$. We~define the homogenization
of $\delta$ by
\begin{gather*}
\widetilde{\delta}=\sum_{i=1}^\ell z^d\, f_i\bigg(\frac{x_1}{z}, \dots, \frac{x_\ell}{z}\bigg)\partial_i.
\end{gather*}
Using the notion of discrete integrals, we define $\zeta_k^m$ as follows.
\begin{gather*}
%\label{eq:zeta}
\zeta_k^m=\sum_{i, j=1}^\ell\left(\sint_{x_i}^{x_j}t^kg(t)^{\underline{m}} \Delta t\right)
\partial_i.
\end{gather*}
Note that the definition of $\zeta_k^m$ is a discrete analogue of~(\ref{eq:eta}).
The next result shows that the homogenizations of these vector fields
form a basis for the Catalan arrangement.

\begin{Theorem}
\label{thm:cat}
$\widetilde{\zeta}_0^m, \widetilde{\zeta}_1^m, \dots, \widetilde{\zeta}_{\ell-2}^m,
\theta_0$ and
the Euler vector field $\theta_E:=z\partial_z+x_1\partial_1+\cdots+x_\ell\partial_\ell$
form a basis of $D(c\Cat_\ell(m))$.
\end{Theorem}

\begin{proof}
Note that the restriction of $c\Cat_\ell(m)$ to the hyperplane $z=0$ is
equal to the braid arrangement $\B_\ell$ (with multiplicity $2m+1$).
In view of Ziegler's characterization of freeness~\cite{zie} (see also~\cite[Corollary~1.35]{yos-survey}), it is sufficient to prove the
following:
\begin{itemize}\itemsep=0pt
\item[$(a)$]
$\widetilde{\zeta}_0^m, \widetilde{\zeta}_1^m, \dots,
\widetilde{\zeta}_{\ell-2}^m\in D(c\Cat_\ell(m))$.
\item[$(b)$]
The restrictions
$\widetilde{\zeta}_0^m|_{z=0}, \widetilde{\zeta}_1^m|_{z=0}, \dots,
\widetilde{\zeta}_{\ell-2}^m|_{z=0}, \theta_0|_{z=0}$
form a basis of $D(\B_\ell, 2m+1)$.
\end{itemize}
First we prove $(b)$. By Proposition~\ref{prop:leading}, we have
\begin{gather*}
\widetilde{\zeta}_0^m|_{z=0}=\eta_k^m.
\end{gather*}
Now $(b)$ is obtained from Proposition~\ref{prop:bm}.
To prove $(a)$, since $\zeta_k^m$ is symmetric, it is sufficient to show that
$\zeta_k^m(x_1-x_2)$ is divisible by $(x_1-x_2-p)$ for $-m\leq p\leq m$.
If $p=0$, it is clear that
\begin{gather*}
\zeta_k^m(x_1-x_2)=
-\ell\sint_{x_2}^{x_1}t^kg(t)^{\underline{m}} \Delta t
\end{gather*}
is divisible by $x_1-x_2$. Now we assume $p>0$. To prove divisibility
by $(x_i-x_j-p)$, we need to show that
\begin{gather}
\label{eq:need}
\sint_{x_2}^{x_2+p}t^kg(t)^{\underline{m}} \Delta t=0.
\end{gather}
Actually, using~(\ref{sumformula}), the left-hand side is equal to
\begin{gather*}
\sum_{t=x_2}^{x_2+p-1}t^kg(t)g(t-1)\cdots g(t-m+1).
\end{gather*}
For each $t=x_2+i$ ($0\leq i\leq p-1$), clearly we have $g(t-i)=0$.
Thus~(\ref{eq:need}) holds.
In the case $-m\leq p<0$, we need to consider
$\sint_{x_1}^{x_2}\Delta t=\sint_{x_1}^{x_1-p}\Delta t$ instead of~(\ref{eq:need}).
The remaining arguments are similar.
\end{proof}

\subsection{A basis for the Shi arrangement}

\looseness=-1
To construct a basis for the Shi arrangement, we need the following.
For $0\leq k\leq\ell$ and $m\geq 0$, let
\begin{gather*}
%\label{eq:gkm}
g_k^{(m)}(t):=
g(t)^{\underline{m-1}}
\prod_{1\leq i<k}(t-x_i+1)
\prod_{k<i\leq\ell}(t-x_i-m+1).
\end{gather*}
We define the following vector field which is a discrete analogue of~(\ref{eq:sigma1})
\begin{gather*}
%\label{eq:tau}
\tau_k^m=
\sum_{i, j=1}^\ell
%\sum_{j=1}^\ell
\left(
\sint_{x_i}^{x_j}g_k^{(m)}(t) \Delta t
\right)
\partial_i.
\end{gather*}
Using these vector fields, we can construct an explicit basis for the Shi arrangement.
\begin{Theorem}
%\label{thm:shi}
$\widetilde{\tau}_1^m-\widetilde{\tau}_2^m, \widetilde{\tau}_2^m-\widetilde{\tau}_3^m, \dots,
\widetilde{\tau}_{\ell-1}^m-\widetilde{\tau}_\ell^m, \theta_0$ and
the Euler vector field
$\theta_E:=z\partial_z+x_1\partial_1+\cdots+x_\ell\partial_\ell$
form a basis for $D(c\Shi_\ell(m))$.
\end{Theorem}

\begin{proof}
The strategy is similar to the proof of Theorem~\ref{thm:cat}. It~is sufficient to prove the
following:
\begin{itemize}\itemsep=0pt
\item[$(a)$]
$\widetilde{\tau}_1^m-\widetilde{\tau}_2^m, \widetilde{\tau}_2^m-\widetilde{\tau}_3^m, \dots,
\widetilde{\tau}_{\ell-1}^m-\widetilde{\tau}_\ell^m\in D(c\Shi_\ell(m))$.
\item[$(b)$]
The restrictions
$(\widetilde{\tau}_1^m-\widetilde{\tau}_2^m)|_{z=0},
(\widetilde{\tau}_2^m-\widetilde{\tau}_3^m)|_{z=0}, \dots,
(\widetilde{\tau}_{\ell-1}^m-\widetilde{\tau}_\ell^m)|_{z=0}$
form a basis of $D(\B_\ell, 2m)$.
\end{itemize}
As in the proof of Theorem~\ref{thm:cat}, $(b)$ is obtained from
Propositions~\ref{prop:leading} and~\ref{prop:simpleroots}.

To prove $(a)$, we need to show that for $1\leq u<v\leq \ell$,
\begin{gather*}
\tau_k^m(x_u-x_v)=\ell\sint_{x_u}^{x_v}g_k^{(m)}(t) \Delta t
\end{gather*}
is divisible by $(x_u-x_v-p)$ for any $1-m\leq p\leq m$.
The case $p=0$ is obvious. We~assume $0<p$. In this case, we need
to show that
\begin{gather*}
\sint_{x_v}^{x_v+p}g_k^{(m)}(t) \Delta t=0.
\end{gather*}
By the formula~(\ref{sumformula}), the left-hand side is equal to
\begin{gather*}
\sum_{s=0}^{p-1}g_k^{(m)}(x_v+s).
\end{gather*}
Suppose $0<p<m$. In this case, we have $m>1$. Then
$g(t)^{\underline{(m-1)}}$ vanishes when $t=x_v+s$, $0\leq s\leq m-2$.
The remaining case is $p=m$. Then $x_u=x_v+m$ and $t=x_v+m-1$.
In this case, the product
\begin{gather*}
\prod_{1\leq i<k}(t-x_i+1)
\prod_{k<i\leq\ell}(t-x_i-m+1)
\end{gather*}
is equal to zero. Indeed, if $v>k$, the factor $(t-x_v-m+1)$ vanishes.
If $v\leq k$, then $u<k$. Then the factor $(t-x_u+1)$ vanishes.

The case $1-m\leq p<0$ is similar (simpler).
\end{proof}

\begin{Remark}
In~\cite{tsu}, generalizations of integral formulas of
Bandlow--Musiker are discussed. It~would be interesting to
consider discrete versions of such generalizations.
\end{Remark}

\subsection*{Acknowledgements}

The authors thank Takuro Abe and Misha Feigin for many useful discussions on the topics.
This work was partially supported by JSPS KAKENHI Grant Numbers JP18H01115, JP15KK0144.

\pdfbookmark[1]{References}{ref}
\LastPageEnding


\begin{thebibliography}{99}
\footnotesize\itemsep=0pt

\bibitem{aefy}
Abe T., Enomoto N., Feigin M., Yoshinaga M., Free reflection multiarrangements
 and quasi-invariants, {i}n preparation.

\bibitem{abe-suy}
Abe T., Suyama D., A basis construction of the extended {C}atalan and {S}hi
 arrangements of the type~{$A_2$}, \href{https://doi.org/10.1016/j.jalgebra.2017.09.024}{\textit{J.~Algebra}} \textbf{493} (2018),
 20--35, \href{https://arxiv.org/abs/1312.5524}{arXiv:1312.5524}.

\bibitem{ath-free}
Athanasiadis C.A., On free deformations of the braid arrangement,
 \href{https://doi.org/10.1006/eujc.1997.0149}{\textit{European~J. Combin.}} \textbf{19} (1998), 7--18.

\bibitem{bm}
Bandlow J., Musiker G., A new characterization for the {$m$}-quasiinvariants of
 {$S_n$} and explicit basis for two row hook shapes, \href{https://doi.org/10.1016/j.jcta.2008.01.011}{\textit{J.~Combin. Theory
 Ser.~A}} \textbf{115} (2008), 1333--1357, \href{https://arxiv.org/abs/0707.3174}{arXiv:0707.3174}.

\bibitem{cha-ves}
Chalykh O.A., Veselov A.P., Commutative rings of partial differential operators
 and {L}ie algebras, \href{https://doi.org/10.1007/BF02125702}{\textit{Comm. Math. Phys.}} \textbf{126} (1990), 597--611.

\bibitem{ede-rei}
Edelman P.H., Reiner V., Free arrangements and rhombic tilings,
 \href{https://doi.org/10.1007/BF02711498}{\textit{Discrete Comput. Geom.}} \textbf{15} (1996), 307--340.

\bibitem{feigin}
Feigin M., Private communication.

\bibitem{fei-ves}
Feigin M., Veselov A.P., Quasi-invariants of {C}oxeter groups and
 {$m$}-harmonic polynomials, \href{https://doi.org/10.1155/S1073792802106064}{\textit{Int. Math. Res. Not.}} \textbf{2002}
 (2002), 521--545, \href{https://arxiv.org/abs/math-ph/0105014}{arXiv:math-ph/0105014}.

\bibitem{fel-ves}
Felder G., Veselov A.P., Action of {C}oxeter groups on {$m$}-harmonic
 polynomials and {K}nizhnik--{Z}amolodchikov equations, \href{https://doi.org/10.17323/1609-4514-2003-3-4-1269-1291}{\textit{Mosc.
 Math.~J.}} \textbf{3} (2003), 1269--1291, \href{https://arxiv.org/abs/math.QA/0108012}{arXiv:math.QA/0108012}.

\bibitem{gpt}
Gao R., Pei D., Terao H., The {S}hi arrangement of the type~{$D_\ell$},
 \href{https://doi.org/10.3792/pjaa.88.41}{\textit{Proc. Japan Acad. Ser.~A Math. Sci.}} \textbf{88} (2012), 41--45,
 \href{https://arxiv.org/abs/1109.1381}{arXiv:1109.1381}.

\bibitem{ot}
Orlik P., Terao H., Arrangements of hyperplanes, \textit{Grundlehren der
 Mathematischen Wissenschaften}, Vol.~300, \href{https://doi.org/10.1007/978-3-662-02772-1}{Springer-Verlag}, Berlin, 1992.

\bibitem{ps-def}
Postnikov A., Stanley R.P., Deformations of {C}oxeter hyperplane arrangements,
 \href{https://doi.org/10.1006/jcta.2000.3106}{\textit{J.~Combin. Theory Ser.~A}} \textbf{91} (2000), 544--597,
 \href{https://arxiv.org/abs/math.CO/9712213}{arXiv:math.CO/9712213}.

\bibitem{sai-log}
Saito K., Theory of logarithmic differential forms and logarithmic vector
 fields, \textit{J.~Fac. Sci. Univ. Tokyo Sect. IA Math.} \textbf{27} (1980),
 265--291.

\bibitem{sai-prim}
Saito K., Period mapping associated to a primitive form, \href{https://doi.org/10.2977/prims/1195182028}{\textit{Publ. Res.
 Inst. Math. Sci.}} \textbf{19} (1983), 1231--1264.

\bibitem{sai-lin}
Saito K., On a linear structure of the quotient variety by a finite reflexion
 group, \href{https://doi.org/10.2977/prims/1195166742}{\textit{Publ. Res. Inst. Math. Sci.}} \textbf{29} (1993), 535--579.

\bibitem{sai-unif}
Saito K., Uniformization of the orbifold of a finite reflection group, in
 Frobenius Manifolds, \textit{Aspects Math.}, Vol.~E36, \href{https://doi.org/10.1007/978-3-322-80236-1_11}{Friedr. Vieweg},
 Wiesbaden, 2004, 265--320.

\bibitem{shi-kl}
Shi J.-Y., The {K}azhdan--{L}usztig cells in certain affine {W}eyl groups,
 \textit{Lecture Notes in Mathematics}, Vol.~1179, \href{https://doi.org/10.1007/BFb0074968}{Springer-Verlag}, Berlin,
 1986.

\bibitem{su-bc}
Suyama D., A basis construction for the {S}hi arrangement of the type
 {$B_\ell$} or {$C_\ell$}, \href{https://doi.org/10.1080/00927872.2013.865051}{\textit{Comm. Algebra}} \textbf{43} (2015),
 1435--1448, \href{https://arxiv.org/abs/1205.6294}{arXiv:1205.6294}.

\bibitem{su-te}
Suyama D., Terao H., The {S}hi arrangements and the {B}ernoulli polynomials,
 \href{https://doi.org/10.1112/blms/bdr118}{\textit{Bull. Lond. Math. Soc.}} \textbf{44} (2012), 563--570,
 \href{https://arxiv.org/abs/1103.3214}{arXiv:1103.3214}.

\bibitem{ter-mult}
Terao H., Multiderivations of {C}oxeter arrangements, \href{https://doi.org/10.1007/s002220100209}{\textit{Invent. Math.}}
 \textbf{148} (2002), 659--674, \href{https://arxiv.org/abs/math.CO/0011247}{arXiv:math.CO/0011247}.

\bibitem{ter-hodge}
Terao H., The {H}odge filtration and the contact-order filtration of
 derivations of {C}oxeter arrangements, \href{https://doi.org/10.1007/s00229-004-0536-z}{\textit{Manuscripta Math.}}
 \textbf{118} (2005), 1--9, \href{https://arxiv.org/abs/math.CO/0205058}{arXiv:math.CO/0205058}.

\bibitem{tsu}
Tsuchida T., On quasiinvariants of {$S_n$} of hook shape, \textit{Osaka~J.
 Math.} \textbf{47} (2010), 461--485, \href{https://arxiv.org/abs/0807.1892}{arXiv:0807.1892}.

\bibitem{yos-prim}
Yoshinaga M., The primitive derivation and freeness of multi-{C}oxeter
 arrangements, \href{https://doi.org/10.3792/pjaa.78.116}{\textit{Proc. Japan Acad. Ser.~A Math. Sci.}} \textbf{78}
 (2002), 116--119, \href{https://arxiv.org/abs/math.CO/0206216}{arXiv:math.CO/0206216}.

\bibitem{yos-char}
Yoshinaga M., Characterization of a free arrangement and conjecture of
 {E}delman and {R}einer, \href{https://doi.org/10.1007/s00222-004-0359-2}{\textit{Invent. Math.}} \textbf{157} (2004), 449--454.

\bibitem{yos-survey}
Yoshinaga M., Freeness of hyperplane arrangements and related topics,
 \href{https://doi.org/10.5802/afst.1413}{\textit{Ann. Fac. Sci. Toulouse Math.}} \textbf{23} (2014), 483--512,
 \href{https://arxiv.org/abs/1212.3523}{arXiv:1212.3523}.

\bibitem{zie}
Ziegler G.M., Multiarrangements of hyperplanes and their freeness, in
 Singularities ({I}owa {C}ity, {IA}, 1986), \textit{Contemp. Math.}, Vol.~90,
 \href{https://doi.org/10.1090/conm/090/1000610}{Amer. Math. Soc.}, Providence, RI, 1989, 345--359.

\end{thebibliography}
\end{document}